\documentclass[preprint,1p]{elsarticle}
\usepackage{amssymb}
\usepackage{amsmath}
\usepackage{amsthm}
\newtheorem{theorem}{Theorem}[section]

\newtheorem{lemma}[theorem]{Lemma}
\newtheorem{remark}[theorem]{Remark}

\journal{Journal of Differential Equations}

\begin{document}

\begin{frontmatter}
\title{The Cauchy Problem for a  One Dimensional Nonlinear Elastic Peridynamic Model}
\author{H. A. Erbay$^1$}
\ead{husnuata.erbay@ozyegin.edu.tr}
\author{A. Erkip$^2$}
\ead{albert@sabanciuniv.edu}
\author{G. M. Muslu$^3$\corref{cor1}}
\ead{gulcin@itu.edu.tr}
\cortext[cor1]{Corresponding author. Tel:
+90 212 285 3257 Fax: +90 212 285 6386}

\address{$^1$ Faculty of Arts and Sciences, Ozyegin  University, Cekmekoy 34794, Istanbul, Turkey}

\address{$^2$ Faculty of Engineering and Natural Sciences, Sabanci University,
        Tuzla 34956,  Istanbul,    Turkey}

\address{$^3$ Department of Mathematics, Istanbul Technical University, Maslak 34469,  Istanbul, Turkey }

\begin{abstract}
This paper studies the Cauchy problem for a one-dimensional nonlinear peridynamic model describing the dynamic
    response of an infinitely long elastic bar. The issues of local well-posedness and smoothness of the solutions
    are discussed. The existence of a global solution  is proved first in the sublinear case and then for
    nonlinearities of degree at most three. The conditions for finite-time blow-up of solutions are established.
\end{abstract}

 \end{frontmatter}

\setcounter{equation}{0}
\section{Introduction}
\label{intro}
In this study, we consider the one-dimensional nonlinear nonlocal partial differential equation, arising in the
peridynamic modelling of an elastic bar,
\begin{equation}
    u_{tt}=\int_{\Bbb R} \alpha (y-x) w(u(y,t)-u(x,t)) dy, ~~~x\in {\Bbb R},~~~t>0 \label{cau1}
\end{equation}
with initial data
\begin{equation}
     u(x,0)=\varphi(x),~~~~~~ u_{t}(x,0)=\psi(x). \label{cau2}
\end{equation}
In (\ref{cau1})-(\ref{cau2}) the subscripts denote partial differentiation, $u=u(x,t)$ is a real-valued function,
the kernel function $\alpha$ is an integrable function on  ${\Bbb R}$ and $w$ is a twice differentiable nonlinear
function with $w(0)=0$.  We  first establish the local well-posedness of the Cauchy problem (\ref{cau1})-(\ref{cau2}),
considering four different cases of initial data: {\it (i)} continuous and bounded functions, {\it (ii)} bounded
$L^{p}$ functions $(1\leq p \leq \infty)$,  {\it (iii)} differentiable and bounded functions and {\it (iv)}  $L^{p}$
functions whose distributional derivatives are also in  $L^{p}({\Bbb R})$. We then extend the results to the case of
$L^{2}$ Sobolev spaces of arbitrary (non-integer) order for the particular form $w(\eta)=\eta^3$.
We prove global existence of solutions for two types of nonlinearities: when $w(\eta)$ is sublinear
and when $w(\eta)=| \eta |^{\nu-1} \eta $ for $\nu \le 3$. Lastly, for the general case we  provide the conditions
under which the solutions of the Cauchy problem blow-up in finite time.

Equation (\ref{cau1}) is a model proposed to describe the dynamical response of an infinite homogeneous elastic bar
within the context of the peridynamic formulation of elasticity theory. The peridynamic theory of solids, mainly
proposed by Silling \cite{silling1}, is an alternative formulation for elastic materials and  has attracted attention
of a growing number of researchers. The most important feature of the peridynamic theory is that the force acting
on a material particle, due to interaction with other particles, is written as a functional of the displacement field.
This means the peridynamic theory is a nonlocal continuum theory and regarding nonlocality it bears a strong resemblance
to more traditional theories of nonlocal elasticity, which are principally based on integral constitutive relations
\cite{kunin,rogula,eringen}. As in other nonlocal theories of elasticity, the main motivation is to propose a generalized
elasticity theory that involves the effect of long-range internal forces of molecular dynamic, neglected in the conventional
theory of elasticity. Another feature of the peridynamic theory is that the peridynamic equation of motion does not involve
spatial derivatives of the displacement field. The absence of spatial derivatives of the displacement field in the
equation of motion makes possible to use the peridynamic equations even at points of displacement discontinuity.
Furthermore, in contrast to the conventional theory of elasticity, the peridynamic theory predicts dispersive wave
propagation as a property of the medium even if the geometry does not define  a length scale.

In the peridynamic theory, by assuming a uniform cross-section and the absence of body forces, the governing equation
of an infinitely long elastic bar is given by
\begin{equation}
   \rho_{0} u_{tt}=\int_{\Bbb R} f(u(y,t)-u(x,t), y-x) dy,   \label{peridynamic1}
\end{equation}
where the axis of the bar coincides with the coordinate axis, a material point on the axis of the bar has coordinate
$x$ in the undeformed state, $u$ and $f$ may be interpreted as averages of the axial displacement and the axial
force located at any $x$ at time $t$, taken over a cross section of the bar, and $\rho_{0}$ is density of the bar
material  \cite{silling2,weckner1}. The space integral in (\ref{peridynamic1}) implies that the displacement at a
generic point is influenced by the  displacements of all particles of the bar (As commonly known, in the conventional theory of elasticity, the equation governing the dynamic response of an infinitely long bar
  is a hyperbolic partial differential
equation that does not involve such a space integral originating from the nonlocal character of the peridynamic theory).
Equation (\ref{peridynamic1}) is obtained by integrating the equation of motion for the axial displacement
over the cross-section and dividing through by the area of the cross-section.  The bar is supposed to be composed of
a homogeneous objective microelastic material \cite{silling1,silling2,weckner1} and its constitutive  behavior is described by the
function $f$. Newton's third law imposes the following restriction on the form of $f$:
\begin{equation}
    f(\eta, \zeta)=-f(-\eta, -\zeta)  \label{thirdrule}
\end{equation}
for all relative displacements $\eta = u(y,t)-u(x,t)$ and relative positions $\zeta = y-x$. For a linear peridynamic
material the constitutive relation is given by
\begin{displaymath}
        f(\eta, \zeta)=\alpha(\zeta)\eta
\end{displaymath}
where $\alpha$ is called the micromodulus function \cite{silling1,silling2,weckner1}. It follows from
(\ref{thirdrule}) that $\alpha$ must be an even function. In
\cite{silling2,weckner1} the dynamic response of a linear peridynamic bar has been investigated and some striking
observations that are not found in the classical theory of elastic bars have been made. Some results on the well-posedness of the Cauchy problem
for the linear peridynamic model have been established in \cite{emmrich1,emmrich2,emmrich3,du}.  In spite of its age, there
is quite extensive literature on the linear peridynamic theory.

It is natural to think that more interesting behavior may be observed when the attention is confined to the nonlinear
peridynamic materials. From this point of view, to the best of our knowledge, the present study
appears to be the first study on  mathematical analysis of nonlinear peridynamic equations. Techniques similar
to those in \cite{duruk1,duruk2,saadet} enable us to  answer some  basic questions, like local well-posedness and lifespan of solutions, as the groundwork of further analysis of the nonlinear peridynamic problem.

In this study we consider the case in which the constitutive behavior is described by a class of nonlinear peridynamic
models in the separable form:
\begin{equation}
   f(\eta, \zeta)=\alpha (\zeta) w(\eta) \label{peridynamic2}
\end{equation}
where $\alpha$ and $w$ are two  functions satisfying the restriction imposed by (\ref{thirdrule}). This
separable form, while allowing us to exploit the properties of convolution-based techniques, is not a serious restriction
and it just makes the proofs easier to follow. Our results can be carried over to the case of general $f(\eta, \zeta)$.
We illustrate this in Theorem \ref{theo2.5}; by imposing certain differentiability
and integrability conditions on $f$, we prove local well-posedness for the general nonlinear peridynamic problem.
Throughout this study we assume that $\alpha$ is an   integrable  even function while $w$ is a differentiable odd function so that (\ref{thirdrule})
is satisfied.

Substitution of the separable form of (\ref{peridynamic2}) into (\ref{peridynamic1}) and non-dimensionalization of the resulting equation
(or simply taking the mass density to be 1) gives the governing equation of the problem in its final form (\ref{cau1})
(Henceforth we use non-dimensional quantities but for convenience use the same symbols). The aim of this study is
three-fold:  to establish the local well-posedness of the Cauchy problem,  to investigate the
existence of a global solution,  and  to present the conditions for finite-time blow-up of solutions.

The paper is organized as follows. In Section 2,  the existence and uniqueness of the local solution for the
nonlinear Cauchy problem is proved by using the contraction mapping principle.  For initial data in fractional Sobolev
spaces the general case seems to involve technical difficulties and  in Section 3 we consider the particular case
$w(\eta)=\eta^{3}$ in the $L^{2}$ Sobolev space setting. We note that the cubic case  can be easily generalized to an arbitrary
polynomial of $\eta$. In Section 4, we consider the issue of global existence versus   finite time blow-up of solutions.
We first show that blow-up must necessarily occur in the $~L^\infty$-norm. We then prove two results on global existence
and finally establish blow-up criteria.

Throughout this paper, $C$ denotes a generic constant. We use  $~\left\Vert u \right\Vert _{\infty}~$ and $~\left\Vert u \right\Vert_{p}~$  to
denote the norms in $L^\infty ({\Bbb R})$ and $L^{p} ({\Bbb R})$ spaces, respectively. The notation
$~\left\langle u, v\right\rangle~$ denotes the inner product in  $L^{2} ({\Bbb R})$.
Furthermore, $C_{b} ({\Bbb R})$ denotes the space of continuous bounded functions on  ${\Bbb R}$, and
$C_{b}^{1} ({\Bbb R})$ is the space of differentiable functions in $C_{b} ({\Bbb R})$ whose first-order derivatives also
belong to $C_{b} ({\Bbb R})$. In the spaces $C_{b} ({\Bbb R})$ and $C_{b}^{1} ({\Bbb R})$ we have the norms
$~\left\Vert u \right\Vert_{\infty}~$ and
$~\left\Vert u \right\Vert_{1,b}=\left\Vert u \right\Vert_{\infty}+\left\Vert u^{\prime} \right\Vert_{\infty}$,
respectively, where the symbol $~^{\prime}$ denotes the differentiation. The Sobolev space $W^{1,p}({\Bbb R})$
is the space of $L^{p}$ functions whose distributional derivatives are also in $L^{p}({\Bbb R})$ with norm
$~\left\Vert u \right\Vert_{W^{1,p}}=\left\Vert u \right\Vert_{p}+\left\Vert u^{\prime} \right\Vert_{p}$. Similarly,
for integer $k\ge 1$, $C_{b}^{k} ({\Bbb R})$  denotes the space of functions whose derivatives
up to order $k$ are continuous and bounded;  $W^{k,p}({\Bbb R})$ denotes the space of $L^{p}$ functions whose derivatives
up to order $k$ are in $L^{p}({\Bbb R})$.

\setcounter{equation}{0}
\section{Local Well Posedness}
\label{sec:2}

Below we will give several versions of local well-posedness of the nonlinear
Cauchy problem given by (\ref{cau1})-(\ref{cau2}). This is achieved in Theorems \ref{theo2.1}-\ref{theo2.4} for four different
cases of initial data spaces, namely $C_{b}({\Bbb R})$, $L^{p} ({\Bbb R})\cap L^\infty ({\Bbb R})$,  $C_{b}^{1}({\Bbb R})$ and
$W^{1,p}({\Bbb R})$ $~(1\leq p \leq \infty)$. The proofs  will follow the same scheme given below.

If  (\ref{cau1}) is integrated twice with respect to $t$, the solution of the Cauchy problem satisfies the integral equation
$u=Su$ where
\begin{equation}
\hspace*{-40pt}
    (Su)(x,t)=\varphi (x)+t\psi (x)+\int_{0}^{t}(t-\tau )(Ku)(x,\tau)d\tau ,  \label{inteq}
\end{equation}
with
\begin{equation}
    (Ku)(x,t)=\int_{\Bbb R }\alpha (y-x)w(u(y,t)-u(x,t))dy.  \label{K}
\end{equation}
Let $X$ be the Banach space with norm $\Vert .\Vert_{X}$, where the initial data lie. We then define the Banach space
$X(T)=C([0,T],X)$, endowed with the norm $\Vert u\Vert_{X(T)}=\max_{t\in [0,T]}\Vert u(t)\Vert_{X}$, and the
closed $R$-ball $Y(T)=\{u\in X(T):\Vert u\Vert_{X(T)}\leq R\}$. We will show that for suitably chosen $R$ and
sufficiently small $T$, the map $S$ is a contraction on $Y(T)$. This will be achieved  by estimating first $Ku$
and then $Su$ in appropriate norms.

In each of the four cases, for $u,v\in Y(T)$ we will get estimates of the form
\begin{equation}
    \left\Vert Su\right\Vert_{X(T)}
        \leq \left\Vert \varphi \right\Vert_{X}+TJ_{1}(R,T)  \label{Sestimate}
\end{equation}
\begin{equation}
    \left\Vert \int_{0}^{t} (t-\tau)((Ku)(\tau)-(Kv)(\tau))d\tau\right\Vert _{X(T)}
        \leq T J_{2}(R,T)\left\Vert u-v\right\Vert_{X(T)}  \label{KLipestimate}
\end{equation}
and hence
\begin{equation}
    \left\Vert Su-Sv\right\Vert _{X(T)}
        \leq TJ_{2}(R,T)\left\Vert u-v\right\Vert_{X(T)}  \label{SLipestimate}
\end{equation}
with certain functions $J_{1}$ and $J_{2}$ nondecreasing in $R$ and $T$. Taking  $R\geq 2\left\Vert \varphi \right\Vert_{X}$
and then choosing $T$  small enough to satisfy
 $T J_{1}(R,T)\leq R/2$  will give $S:Y(T)\rightarrow Y(T)$; the further choice
$T J_{2}(R,T)\leq 1/2$ will show that $S$ is a contraction. This implies that there is a unique $u\in Y(T)$
satisfying the integral equation $u=Su$. But, as $Ku$ is clearly continuous in $t$, we can
differentiate (\ref{inteq}) to get
\begin{displaymath}
    u_{t}(x,t)=\psi (x)+\int_{0}^{t}(Ku)(x,\tau )d\tau  \label{ut}
\end{displaymath}
and consequently $~ u_{tt}(x,t)=(Ku)(x,t)$. This shows that $u\in C^{2}([0,T],X)$ solves (\ref{cau1})-(\ref{cau2}). Finally,
if  $u_{1}$ and $u_{2}$ satisfy (\ref{cau1})-(\ref{cau2}) with initial data $\varphi_{i}, \psi_{i}$ \ for $i=1,2$ we get
\begin{displaymath}
u_{1}-u_{2}=\varphi_{1}-\varphi_{2}+t(\psi_{1}-\psi_{2}) + \int_{0}^{t} (t-\tau) ((K{u_1})(\tau)-(K{u_2})(\tau))d\tau.
\end{displaymath}
Then the estimate (\ref{KLipestimate}) shows that
\begin{displaymath}
    \left\Vert u_{1}-u_{2}\right\Vert _{X(T)}
        \leq \left\Vert \varphi_{1}-\varphi_{2}\right\Vert_{X}+t\left\Vert \psi_{1}-\psi_{2}\right\Vert_{X}
        +TJ_{2}(R,T)\left\Vert u_{1}-u_{2}\right\Vert_{X(T)}.
\end{displaymath}
When $TJ_{2}(R,T) \leq 1/2$,
\begin{displaymath}
    \left\Vert u_{1}-u_{2}\right\Vert _{X(T)}
        \leq 2\left\Vert \varphi_{1}-\varphi_{2}\right\Vert_{X}+2t\left\Vert \psi_{1}-\psi_{2}\right\Vert_{X}
\end{displaymath}
for $t\in [0,T]$. This shows that, locally, solutions of (\ref{cau1})-(\ref{cau2})
depend continuously on initial data; thus the problem (\ref{cau1})-(\ref{cau2}) is locally well posed.

The Mean Value Theorem for nonlinear estimates  and the following lemma for convolution
estimates will be our main tools:
\begin{lemma} \label{lem2.1}
    Let $1\leq p\leq \infty $ and $f\in L^{1}({\Bbb R})$, $g\in L^{p}({\Bbb R})$. The convolution
    $(f\ast g)(x)=\int_{\Bbb R }f(y-x)g(y)dy$ is well defined and $f\ast g\in L^{p}({\Bbb R})$ with
    \begin{displaymath}
        \Vert f\ast g\Vert _{p}\leq \Vert f\Vert _{1}\Vert g\Vert _{p}.
    \end{displaymath}
\end{lemma}
In the estimates below, we will often encounter the nondecreasing function $M(R)$ defined for $R>0$ as
\begin{equation}
    M(R)=\max_{\left\vert \eta \right\vert \leq 2R}\left\vert w^{\prime }\left(\eta \right) \right\vert.  \label{MR}
\end{equation}

We now state and prove (i.e. show that the estimates (\ref{Sestimate}) and (\ref{SLipestimate}) hold) the four theorems of local well posedness.

\begin{theorem}\label{theo2.1}
    Assume that $\alpha \in L^{1}({\Bbb R})$  and $w\in C^{1}({\Bbb R})$ with $w(0)=0$. Then there is some $~T>0~$ such that the Cauchy problem
    (\ref{cau1})-(\ref{cau2}) is well posed with solution in $~C^{2}([0,T], C_{b}({\Bbb R}))~$ for initial data
    $~\varphi, \psi \in C_{b}({\Bbb R})$.
\end{theorem}
\begin{proof}
    Take $X=C_{b}({\Bbb R})$. For $u\in Y(T)$,  clearly $Ku$ is continuous in $x$ and $t$ and hence $Su\in C^{2}([0,T],X)$.
    Since $w(0)=0$ and
    \begin{displaymath}
    \left\vert u(y,t)-u(x,t)\right\vert \leq 2\Vert u(t)\Vert_{\infty},
    \end{displaymath}
    the Mean Value Theorem implies
    \begin{displaymath}
        |w(u(y)-u(x))|\leq \sup_{|\eta |\leq 2\Vert u\Vert_{\infty }}|w^{\prime}(\eta)|~|u(y)-u(x)|
        = M(\Vert u\Vert_{\infty })(|u(y)|+|u(x)|),
    \end{displaymath}
    where we have suppressed the $t$ variable for convenience. Then
    \begin{eqnarray}
        \left\vert (Ku)(x,t)\right\vert
        &\leq & M(\Vert u(t)\Vert_{\infty})\int_{\Bbb R}
            \left\vert \alpha (y-x)\right\vert (|u(y,t)|+|u(x,t)|)dy  \nonumber \\
        &=& M(\Vert u(t)\Vert_{\infty})\left[ \left( \left\vert
            \alpha \right\vert \ast \left\vert u\right\vert \right)(x,t)+\Vert \alpha \Vert_{1}|u(x,t)|\right],
            \label{Kbound}
    \end{eqnarray}
    and
    \begin{equation}
    \left\Vert (Ku)(t)\right\Vert_{\infty }
        \leq 2 M(\Vert u(t)\Vert_{\infty})\Vert \alpha \Vert_{1}
            \Vert u(t)\Vert_{\infty} \label{kuinf}
    \end{equation}
    where we have used Lemma \ref{lem2.1}.
    Then
    \begin{displaymath}
        \left\vert (Su)(x,t)\right\vert
        \leq \left\vert \varphi (x)\right\vert +t\left\vert \psi (x)\right\vert
            +\int_{0}^{t}(t-\tau)\left\vert (Ku)(x,\tau)\right\vert d\tau,
    \end{displaymath}
    and
    \begin{equation}
        \left\Vert (Su)(t)\right\Vert_{\infty }
        \leq \Vert \varphi \Vert_{\infty}+t\Vert \psi \Vert_{\infty}
            +2\Vert \alpha \Vert_{1}\int_{0}^{t}(t-\tau )M(\Vert u(\tau)\Vert_{\infty})\Vert u(\tau)\Vert_{\infty}d\tau.
            \label{265}
    \end{equation}
     As $u\in Y(T)$, this gives $M(\Vert u(\tau)\Vert _{\infty})\leq M(R)$ and hence
    \begin{eqnarray}
        \left\Vert Su\right\Vert_{X(T)}
        &\leq & \Vert \varphi \Vert_{\infty}+T\Vert \psi \Vert_{\infty }
            +2M(R)\Vert \alpha \Vert_{1}\Vert u\Vert_{X(T)}\sup_{t\in [0,T]}\int_{0}^{t}(t-\tau )d\tau  \nonumber \\
        &\leq & \Vert \varphi \Vert_{\infty}+T\Vert \psi \Vert_{\infty}+M(R)R\Vert \alpha \Vert_{1}T^{2}.
        \label{infestimate}
    \end{eqnarray}
    This proves (\ref{Sestimate}) with $J_{1}(R,T)=\Vert \psi \Vert_{\infty}+M(R)R\Vert \alpha \Vert_{1}T$.
    Now let $u, v\in Y(T)$. We start by estimating $Ku-Kv$. Again suppressing $t$,
    \begin{displaymath}
        |w(u(y)-u(x))-w(v(y)-v(x))|\leq M(R)(|u(y)-v(y)|+|u(x)-v(x)|),
    \end{displaymath}
    and
    \begin{eqnarray}
        \left\vert (Ku)(x,t)-(Kv)(x,t)\right\vert
        &\leq & M(R)\left( \left\vert\alpha \right\vert \ast \left\vert u-v\right\vert \right)(x,t) \nonumber  \\
        & & +M(R)\Vert \alpha \Vert_{1}|u(x,t)-v(x,t)|. \label{Klipinf}
    \end{eqnarray}
    Similar to (\ref{infestimate}) we get
    \begin{equation}
        \left\Vert (Su)(t)-(Sv)(t)\right\Vert_{\infty }
        \leq  2M(R)\Vert \alpha \Vert_{1}\int_{0}^{t}(t-\tau )\Vert u(\tau )-v(\tau )\Vert_{\infty }~d\tau  \label{285}
    \end{equation}
    and
    \begin{equation}
        \left\Vert Su-Sv\right\Vert_{X(T)}\leq M(R)\Vert \alpha \Vert_{1}T^{2}\Vert u-v\Vert_{X(T)} \label{inflipestimate}
    \end{equation}
    which proves (\ref{SLipestimate}) with $J_{2}(R,T)=M(R)\Vert \alpha \Vert_{1}T$. According to the scheme described above,
    this completes the proof.
\end{proof}

\begin{theorem}\label{theo2.2}
    Let $1\leq p\leq \infty$. Assume that $\alpha \in L^{1}({\Bbb R})$  and $w\in C^{1}({\Bbb R})$ with $w(0)=0$. Then there is some $~T>0~$ such that
    the Cauchy problem (\ref{cau1})-(\ref{cau2}) is well posed with solution in
    $~C^{2}\left( [0,T], L^{p}({\Bbb R})\cap L^{\infty}({\Bbb R})\right)~$ for initial data
    $~\varphi, \psi \in L^{p}({\Bbb R})\cap L^{\infty}({\Bbb R})$.
\end{theorem}
\begin{proof}
    Let $X=L^{p}({\Bbb R})\cap L^{\infty}({\Bbb R})$ with norm
    $\Vert u\Vert_{X}=\Vert u\Vert_{p}+\Vert u\Vert_{\infty}$. As we already have the $L^{\infty }$ estimates given in
    (\ref{265}) and (\ref{285}), we now look for the corresponding $L^{p}$ estimates.  Lemma \ref{lem2.1} implies
    $\left\Vert \left( \left\vert \alpha \right\vert \ast \left\vert u\right\vert \right) (t)\right\Vert_{p}
    \leq \Vert \alpha \Vert _{1}|\left\Vert u(t)\right\Vert_{p}$ so
    \begin{equation}
    \left\Vert (Ku)(t)\right\Vert_{p }
            \leq 2 M(\Vert u(t)\Vert_{\infty})\Vert \alpha \Vert_{1}
            \Vert u(t)\Vert_{p} \label{kup}
    \end{equation}
     and Minkowski's inequality for integrals will yield
    \begin{equation}
        \left\Vert (Su)(t)\right\Vert_{p}
        \leq \Vert \varphi \Vert_{p}+t\Vert \psi \Vert_{p}+2\Vert
            \alpha \Vert_{1}\int_{0}^{t}(t-\tau )M(\Vert u(\tau)\Vert_{\infty})\Vert u(\tau)\Vert_{p}~d\tau.
            \label{210}
    \end{equation}
    Adding this to  the $L^{\infty}$ estimate  (\ref{265}), we get
    \begin{displaymath}
        \left\Vert Su\right\Vert_{X(T)}
        \leq \Vert \varphi \Vert_{X}+T \Vert \psi \Vert_{X}+M(R)R\Vert \alpha \Vert_{1}T^{2}.
    \end{displaymath}
    Similarly we have
    \begin{equation}
        \left\Vert (Su)(t)-(Sv)(t)\right\Vert_{p}
        \leq 2M(R)\Vert \alpha \Vert_{1}\int_{0}^{t}(t-\tau )\Vert u(\tau )-v(\tau)\Vert_{p}~d\tau.
            \label{211}
    \end{equation}
  Adding this to (\ref{285}) gives
    \begin{displaymath}
        \left\Vert Su-Sv\right\Vert_{X(T)}
        \leq M(R)\Vert \alpha \Vert_{1}T^{2}\Vert u-v\Vert _{X(T)}
    \end{displaymath}
    and concludes the proofs of (\ref{Sestimate}) and (\ref{SLipestimate}).
\end{proof}
\begin{theorem}\label{theo2.3}
    Assume that $\alpha \in L^{1}({\Bbb R})$  and $w\in C^{2}({\Bbb R})$ with $w(0)=0$. Then there is some $~T>0~$ such that the Cauchy problem
    (\ref{cau1})-(\ref{cau2}) is well posed with solution in $~C^{2}([0,T], C_{b}^{1}({\Bbb R}))~$ for
    initial data $~\varphi ,\psi \in C_{b}^{1}({\Bbb R})$.
\end{theorem}
\begin{proof}
    We now take $X=C_{b}^{1}({\Bbb R})$ for which the norm is
    $\Vert u\Vert_{1,b}=\Vert u\Vert_{\infty}+\Vert u^{\prime}\Vert_{\infty}$. Since we have the sup norm
    estimates (\ref{265}) and (\ref{285})  all we need is estimates for their $x$ derivatives.
    Throughout this proof we will suppress $t$ (or $\tau$) to keep the expressions shorter, whenever it is clear from the
    context. Differentiating (\ref{K}) gives
    \begin{eqnarray*}
        \frac{\partial }{\partial x}(Ku)(x)
        &=& \frac{\partial }{\partial x}\int_{\Bbb R}\alpha (y-x)w(u(y)-u(x))dy \\
        &=& \frac{\partial }{\partial x}\int_{\Bbb R}\alpha (z)w(u(x+z)-u(x))dz \\
        &=& \int_{\Bbb R}\alpha (z)w^{\prime}(u(x+z)-u(x))(u_{x}(x+z)-u_{x}(x))dz \\
        &=& \int_{\Bbb R}\alpha (y-x)w^{\prime}(u(y)-u(x))(u_{x}(y)-u_{x}(x))dy.
    \end{eqnarray*}
    Recall that $|w^{\prime}(u(y)-u(x))|\leq M(\Vert u(t)\Vert_{\infty})$ due to (\ref{MR}). Then
    \begin{eqnarray}
        |(Ku)_{x}(x)|
        &\leq & M(\Vert u\Vert_{\infty})\int_{\Bbb R}\left\vert \alpha (y-x)
            \right\vert (|u_{x}(y)|+|u_{x}(x)|)dy  \nonumber \\
        &\leq & M(\Vert u\Vert_{\infty})[\left( |\alpha |\ast |u_{x}|\right) (x)
            +\Vert \alpha \Vert_{1}|u_{x}(x)|].  \label{Kx}
    \end{eqnarray}
    Since
    \begin{equation}
        \left\vert (Su)_{x}(x,t)\right\vert
        \leq  \left\vert \varphi^{\prime}(x)\right\vert
            +t\left\vert \psi ^{\prime }(x)\right\vert +\int_{0}^{t}(t-\tau )~|(Ku)_{x}(x,\tau)|d\tau ,  \label{sux} \
    \end{equation}
    we have
    \begin{displaymath}
        \left\Vert (Su)_{x}(t)\right\Vert_{\infty }
        \leq  \Vert \varphi^{\prime}\Vert _{\infty }+t\Vert \psi^{\prime}\Vert_{\infty}
            +2 \Vert \alpha \Vert_{1}\int_{0}^{t}(t-\tau)M(\Vert u(\tau)\Vert_{\infty })\Vert u_{x}(\tau)\Vert_{\infty }d\tau.
    \end{displaymath}
    But $M(\Vert u(\tau)\Vert_{\infty })\leq M(R)$ so adding up with the estimate (\ref{265}) proves (\ref{Sestimate})
    \begin{eqnarray*}
        \left\Vert Su\right\Vert_{X(T)}
        & = & \max_{t\in [ 0,T]}(\Vert (Su)(t)\Vert_{\infty}+\Vert (Su)_{x}(t)\Vert_{\infty}) \\
        &\leq & \Vert \varphi \Vert_{1,b}
            +T\Vert \psi \Vert_{1,b}+M(R)R\Vert \alpha \Vert_{1}T^{2}.
    \end{eqnarray*}
    Next, for $\left\vert \eta_{i}\right\vert \leq 2R$ and $\left\vert \mu_{i}\right\vert \leq 2R$ for $(i=1,2)$, we estimate
    \begin{eqnarray*}
        \left\vert w^{\prime}\left( \eta_{1}\right) \mu_{1}-w^{\prime}\left(\eta_{2}\right) \mu_{2}\right\vert
        &\leq & \left\vert w^{\prime}\left( \eta_{1}\right) \right\vert \left\vert \mu_{1}-\mu_{2}\right\vert
            +\left\vert w^{\prime}\left( \eta_{1}\right) -w^{\prime}\left( \eta_{2}\right) \right\vert \left\vert \mu_{2}\right\vert \\
        &\leq & M(R)\left\vert \mu_{1}-\mu_{2}\right\vert
            +2R\max_{\eta \leq 2R}\left\vert w^{\prime\prime}\left( \eta \right) \right\vert \left\vert\eta_{1}-\eta_{2}\right\vert \\
        &\leq & M(R)\left\vert \mu_{1}-\mu_{2}\right\vert +2RN(R)\left\vert \eta_{1}-\eta_{2}\right\vert
        \end{eqnarray*}
        where $N(R)=\max_{\eta \leq 2R}\left\vert w^{\prime\prime}\left( \eta \right) \right\vert$. Then
    \begin{eqnarray}
        |(Ku-Kv)_{x}(x)|
        &\leq & M(R)\int_{\Bbb R}\left\vert \alpha (y-x)
            \right\vert (|u_{x}(y)-v_{x}(y)|+|u_{x}(x)-v_{x}(x)|)dy  \nonumber \\
        & & +2RN(R)\int_{\Bbb R}\left\vert \alpha (y-x)\right\vert (|u(y)-v(y)|+|u(x)-v(x)|)dy  \nonumber \\
    &\leq & M(R)(\left( |\alpha |\ast |u_{x}-v_{x}|\right) (x)+\Vert \alpha \Vert_{1}|u_{x}(x)-v_{x}(x)|) \nonumber  \\
    & & +2RN(R)(\left( |\alpha |\ast |u-v|\right) (x)+\Vert \alpha \Vert_{1}|u(x)-v(x)|)  \label{Lip_x}
    \end{eqnarray}
    and
    \begin{eqnarray}
        \Vert (Su-Sv)_{x}(t)\Vert _{\infty}
        &\leq & 2M(R)\Vert \alpha \Vert_{1}\int_{0}^{t}(t-\tau)\Vert u_{x}(\tau)-v_{x}(\tau)\Vert_{\infty}d\tau  \nonumber \\
        & & +4RN(R)\Vert \alpha \Vert_{1}\int_{0}^{t}(t-\tau)\Vert u(\tau)-v(\tau)\Vert_{\infty}d\tau  \nonumber \\
        &\leq & \left( M(R)+2RN(R)\right) \Vert \alpha \Vert_{1}T^{2}\Vert u-v\Vert_{X(T)}   \label{LipS_x}
    \end{eqnarray}
    Finally, adding this to  (\ref{285}) we get (\ref{SLipestimate}) in the form
    \begin{eqnarray*}
        \left\Vert Su-Sv\right\Vert_{X(T)}
        &\leq & \max_{t\in [0,T]}(\Vert (Su-Sv)(t)\Vert_{\infty}+\Vert (Su-Sv)_{x}(t)\Vert_{\infty}) \\
        &\leq & 2 \left( M(R)+RN(R)\right) \Vert \alpha \Vert_{1}T^{2}\Vert u-v\Vert_{X(T)}.
    \end{eqnarray*}
\end{proof}

\begin{theorem}\label{theo2.4}
    Let $1\leq p\leq \infty$. Assume that $\alpha \in L^{1}({\Bbb R})$  and $w\in C^{2}({\Bbb R})$ with $w(0)=0$. Then there is some $~T>0~$ such
    that the Cauchy problem (\ref{cau1})-(\ref{cau2}) is well posed with solution in
    $~C^{2}([0,T], W^{1,p}({\Bbb R}))~$ for initial data $~\varphi, \psi \in W^{1,p}({\Bbb R})$.
\end{theorem}
\begin{proof}
    Let $X=W^{1,p}({\Bbb R})\subset L^{\infty }({\Bbb R})$. Since
    $~\left\Vert u \right\Vert_{W^{1,p}}=\left\Vert u \right\Vert_{p}+\left\Vert u^{\prime} \right\Vert_{p}$,
    we need derivative estimates only  in addition to the $L^{p}$ estimates (\ref{210}) and (\ref{211}).
    For  $u, v\in Y(T)$, from (\ref{Kx})-(\ref{sux}) and  Minkowski's inequality we have
    \begin{displaymath}
        \left\Vert (Su)_{x}(t)\right\Vert_{p}
        \leq  \Vert \varphi^{\prime}\Vert_{p }+t\Vert \psi^{\prime}\Vert_{p}
            +2 \Vert \alpha \Vert_{1}\int_{0}^{t}(t-\tau)M(\Vert u(\tau)\Vert_{\infty })\Vert u_{x}(\tau)\Vert_{p }d\tau.
    \end{displaymath}
    We note that the term $\Vert u\Vert_{\infty}$  can be eliminated by using
    $\Vert u\Vert_{\infty} \le C \Vert u\Vert_{W^{1,p}}$ due to the Sobolev Embedding Theorem.
    So $M(\Vert u(\tau)\Vert_{\infty })\leq M(CR)$ and adding up the above estimate with  (\ref{210}) proves (\ref{Sestimate});
    \begin{eqnarray*}
        \left\Vert Su\right\Vert_{X(T)}
        & = & \max_{t\in [ 0,T]}(\Vert (Su)(t)\Vert_{p}+\Vert (Su)_{x}(t)\Vert_{p}) \\
        &\leq & \Vert \varphi \Vert_{W^{1,p}}
            +T \Vert \psi \Vert_{W^{1,p}}+M(CR)R\Vert \alpha \Vert_{1}T^{2}.
    \end{eqnarray*}
    Again from (\ref{Lip_x}) we get
    \begin{eqnarray*}
        \Vert (Su-Sv)_{x}(t)\Vert_{p}
        &\leq & 2M(CR)\Vert \alpha \Vert_{1}\int_{0}^{t}(t-\tau)\Vert u_{x}(\tau)-v_{x}(\tau)\Vert_{p}d\tau \\
        & & +4RN(CR)\Vert \alpha \Vert_{1}\int_{0}^{t}(t-\tau)\Vert u(\tau)-v(\tau)\Vert_{p}d\tau.
    \end{eqnarray*}
    Together with (\ref{211}), we conclude the proof:
    \begin{displaymath}
        \Vert Su-Sv\Vert_{X(T)}
        \leq  2\left( M(CR)+RN(CR)\right) \Vert \alpha \Vert_{1}T^{2}\Vert u-v\Vert_{X(T)}.
    \end{displaymath}
\end{proof}
\begin{remark}\label{rem2.1}
    We remark that the investigation  can also continue for smoother data in  along the same lines. That is, for initial
    data in $C_{b}^{k}({\Bbb R})$ or $W^{k,p}({\Bbb R})$ with integer $k$ we can prove higher-order versions of
    Theorems \ref{theo2.3}-\ref{theo2.4}. Also,
    the proofs clearly indicate that in Theorems \ref{theo2.1} and \ref{theo2.2}  we can replace the assumption
    $w\in C^{1}({\Bbb R})$ with its weaker form: $w$  is locally Lipschitz. Similarly, in Theorems
    \ref{theo2.3} and \ref{theo2.4} the assumption  $w\in C^{2}({\Bbb R})$ can be weakened to the condition:
    $w^{\prime}$ is locally Lipschitz.
\end{remark}
\begin{remark}
    The above theorems of  local well-posedness  can be easily adapted to the general peridynamic equation
    (\ref{peridynamic1}).
    Theorem \ref{theo2.5} below extends Theorem \ref{theo2.1} to the general peridynamic equation (\ref{peridynamic1}). Clearly, similar extensions are
    also possible in the cases of Theorems \ref{theo2.2}-\ref{theo2.4}.
\end{remark}
\begin{theorem}\label{theo2.5}
    Assume that $ f\left(\zeta, 0\right) =0$ and $ f\left( \zeta, \eta \right) $ is continuously differentiable in
    $\eta $ for almost all $\zeta$. Moreover, suppose that for each $R>0$, there are integrable functions
    $\Lambda^{R}_{1}$, $\Lambda^{R}_{2}$  satisfying
    \begin{displaymath}
         \left\vert f\left( \zeta, \eta \right) \right\vert \leq \Lambda^{R}_{1}\left( \zeta \right), ~~~~~~
         \left\vert f_{\eta}\left( \zeta, \eta\right) \right\vert \leq \Lambda^{R}_{2}\left( \zeta \right)
    \end{displaymath}%
    for almost all $\zeta$ and for all $\left\vert \eta \right\vert \leq 2R$. Then there is some $~T>0~$ such that the Cauchy problem (\ref{peridynamic1})-(\ref{cau2})
    is well posed with solution in $~C^{2}([0,T], C_{b}({\Bbb R}))~$ for initial data
    $~\varphi, \psi \in C_{b}({\Bbb R})$.
\end{theorem}
\begin{proof}
    We proceed as in the proof of Theorem \ref{theo2.1}.  By the Dominated Convergence Theorem, the condition
    $\left\vert f\left( \zeta, \eta \right)\right\vert \leq \Lambda^{R}_{1}\left( \zeta \right) $ implies that $Ku$ is
    continuous in $x$ so that $S:X(T)\rightarrow X(T)$.  Using the second inequality
    $\left\vert f_{\eta }\left( \zeta, \eta \right) \right\vert \leq \Lambda^{R}_{2}\left( \zeta \right) $ the
    estimates for  $\left\Vert (Su)(t)\right\Vert_{\infty }$ and
    $\left\Vert (Su)(t)-(Sv)(t)\right\Vert_{\infty }$ follow as in (\ref{infestimate}) and (\ref{inflipestimate}), just
    replacing the term $M(R)\Vert \alpha \Vert_{1}$ by $\Vert \Lambda^{R}_{1}\Vert_{1}$ and
    $\Vert \Lambda^{R}_{2}\Vert_{1}$ respectively, completing the proof.
\end{proof}
\begin{remark}
    To finish this section let us briefly mention the issue of multidimensional case in the general three-dimensional
    peridynamic theory.     Although our analysis in this section has been presented for the one-dimensional case of the
    peridynamic formulation, the techniques used can be extended to the case of a system of three peridynamic equations
    in three space variables without any additional complication. Namely, if we replace
    the scalars  $x$, $y$, $u$,  $w$ and $\alpha$ in (\ref{cau1})-(\ref{cau2}) by the vectors $\bf x$, $\bf y$, $\bf u$,
    $\bf w({\bf u})$ and the matrix $\bf \alpha({\bf x})$, respectively, the local existence theorems given above
    will still  be valid.
\end{remark}

\setcounter{equation}{0}
\section{The Cubic Nonlinear Case in $H^{s}({\Bbb R})$}
\label{sec:3}
We now want to consider the Cauchy problem (\ref{cau1})-(\ref{cau2}) in the $L^{2}$ Sobolev space setting. We will
denote the $L^{2}$ Sobolev space of order $s$ on $\Bbb R$ by $H^{s}(\Bbb R)$ with norm
\begin{displaymath}
    \left\Vert u \right\Vert_{H^{s}}^{2}=\int_{\Bbb R} (1+ \xi^{2})^{s}|\widehat{u}(\xi)|^{2}  \mbox{d}\xi
\end{displaymath}
where $\widehat{u}$ denotes  the Fourier transform of $u$. For integer $k\ge 0$, $H^{k}({\Bbb R})=W^{k,2}(\Bbb R)$.

As mentioned in Remark \ref{rem2.1}, the proof in the case of $ H^{1}({\Bbb R})$ can be extended to
$H^{k}({\Bbb R})$. On the other hand, for non-integer $s$, $H^{s}$ estimates of the nonlinear term $w(u(y)-u(x))$ involve technical difficulties. Nevertheless, the case of polynomial nonlinearities can be handled in a
straightforward manner. We illustrate this in the typical case $w(\eta)=\eta^3$. Then, the integral on the right-hand
side of (\ref{cau1}) can be computed explicitly in terms of convolutions and the Cauchy problem (\ref{cau1})-(\ref{cau2})
becomes
\begin{eqnarray}
    \hspace*{-20pt} && u_{tt}= \alpha \ast {u^3}-3u (\alpha \ast {u^2})+3{u^2} (\alpha \ast {u})-Au^{3} \label{peridynamic3}\\
    \hspace*{-20pt} && u(x,0)=\varphi(x),~~~~~~ u_{t}(x,0)=\psi(x), \label{peridynamic4}
\end{eqnarray}
where $A=\int_{\Bbb R} \alpha (y)dy$.

For the estimates below we need the following lemmas.
\begin{lemma}\label{lem3.1}
    Let $\alpha \in L^{1}({\Bbb R})$ and $u \in H^{s}({\Bbb R})$ for $s\ge 0$. Then $\alpha \ast u \in H^{s}({\Bbb R})$ and
    \begin{displaymath}
        \left\Vert \alpha \ast u \right\Vert_{H^{s}} \leq \left\Vert \alpha \right\Vert_{1}\left\Vert u \right\Vert_{H^{s}}.
    \end{displaymath}
\end{lemma}
\begin{lemma}\label{lem3.2} \cite{taylor}
    Let $s\ge 0$ and $u, v \in H^{s}({\Bbb R}) \cap L^{\infty}({\Bbb R})$. Then $uv \in H^{s}({\Bbb R})$ and  for some
    constant $C$ (independent of $u$ and $v$)
    \begin{displaymath}
        \left\Vert u v \right\Vert_{H^{s}} \leq C( \left\Vert u \right\Vert_{\infty}\left\Vert v \right\Vert_{H^{s}}
            +\left\Vert v \right\Vert_{\infty}\left\Vert u \right\Vert_{H^{s}}).
    \end{displaymath}
\end{lemma}
For the space $ H^{s}({\Bbb R})\cap L^{\infty}({\Bbb R})$ we use the norm
$\left\Vert u \right\Vert_{s,\infty}=\left\Vert u \right\Vert_{H^{s}}+\left\Vert u \right\Vert_{\infty}$.
In general, Lemma \ref{lem3.2}  implies that $ H^{s}({\Bbb R})\cap L^{\infty}({\Bbb R})$ is
an algebra
\begin{equation}
    \left\Vert u v \right\Vert_{s,\infty} \leq C \left\Vert u \right\Vert_{s,\infty}\left\Vert v \right\Vert_{s,\infty},
          \label{uv}
\end{equation}
and, by Lemmas \ref{lem2.1} and \ref{lem3.1}, for $\alpha \in L^{1}({\Bbb R})$
\begin{equation}
    \left\Vert \alpha \ast u \right\Vert_{s,\infty} \leq \left\Vert \alpha \right\Vert_{1}\left\Vert u \right\Vert_{s,\infty}.
            \label{betacon}
\end{equation}
We are now ready to prove the following theorem.
\begin{theorem}\label{theo3.1}
    Let $~s>0$. Assume that $~\varphi, \psi \in H^{s}({\Bbb R})\cap L^{\infty}({\Bbb R})$. Then there is some $~T>0~$ such that the
    Cauchy problem (\ref{peridynamic3})-(\ref{peridynamic4}) is well posed with solution in
    $~C^{2}([0,T], H^{s}({\Bbb R})\cap L^{\infty}({\Bbb R}))~$.
\end{theorem}
\begin{proof}
    We follow the scheme summarized at the beginning of  Section 2 for
    $X=H^{s}({\Bbb R})\cap L^{\infty }({\Bbb R})$. Explicitly,
    \begin{displaymath}
        Ku=\alpha \ast u^{3}-3u(\alpha \ast u^{2})+3u^{2}(\alpha \ast u)-Au^{3}.
    \end{displaymath}
    We start by estimating the terms of the form $u^{i}(\alpha \ast u^{j})$ for $i+j=3$. Clearly from (\ref{uv}) and
    (\ref{betacon}),
    $\left\Vert u^{i}\left( \alpha \ast u^{j}\right)\right\Vert_{s,\infty }
    \leq C \left\Vert \alpha \right\Vert_{1}\left\Vert u\right\Vert_{s,\infty }^{3}$.
    Nevertheless, for later use we derive a more precise estimate. By repeated use of Lemma \ref{lem3.2} we have
    $\left\Vert u^{j}\right\Vert_{H^{s}}\leq C_{j}\left\Vert u\right\Vert_{\infty}^{j-1}\left\Vert u\right\Vert_{H^{s}}$.
    Again, by Lemmas \ref{lem3.1} and \ref{lem3.2}
    \begin{eqnarray*}
        \left\Vert u^{i}\left( \alpha \ast u^{j}\right) \right\Vert_{H^{s}}
            &\leq & C (\left\Vert u^{i}\right\Vert_{H^{s}}\left\Vert \alpha \ast u^{j}\right\Vert_{\infty }
                    +\left\Vert u^{i}\right\Vert_{\infty }\left\Vert \alpha \ast u^{j}\right\Vert _{H^{s}}) \\
            &\leq & C \left( C_{i}+C_{j}\right) \left\Vert \alpha \right\Vert_{1}\left\Vert u\right\Vert_{\infty}^{2}\left\Vert u\right\Vert_{H^{s}},
    \end{eqnarray*}
    so that
    \begin{displaymath}
        \left\Vert Ku\right\Vert_{s,\infty}\leq C \left\Vert \alpha \right\Vert_{1} \left\Vert u\right\Vert_{\infty}^{2}\left\Vert u\right\Vert_{s,\infty }.
    \end{displaymath}
    Similarly
    \begin{eqnarray*}
        \left\Vert u^{i}(\alpha \ast u^{j})-v^{i}(\alpha \ast v^{j})\right\Vert_{s,\infty}
        &\leq &\left\Vert u^{i}(\alpha \ast (u^{j}-v^{j}))\right\Vert_{s,\infty}+\left\Vert (u^{i}-v^{i})(\alpha \ast v^{j})\right\Vert_{s,\infty} \\
        &\leq & C \left(\left\Vert u^{i}\right\Vert_{s,\infty }\left\Vert \alpha \ast (u^{j}-v^{j})\right\Vert _{s,\infty}
                +\left\Vert u^{i}-v^{i}\right\Vert_{s,\infty }\left\Vert \alpha \ast v^{j}\right\Vert_{s,\infty}\right) \\
        &\leq & C \left\Vert \alpha \right\Vert _{1}\left( \left\Vert u^{i}\right\Vert_{s,\infty}\left\Vert u^{j}-v^{j}\right\Vert_{s,\infty}
                +\left\Vert v^{j}\right\Vert_{s,\infty}\left\Vert u^{i}-v^{i}\right\Vert_{s,\infty}\right)  \\
        &\leq &\left\Vert \alpha \right\Vert_{1}P\left( \left\Vert u\right\Vert_{s,\infty},\left\Vert v
            \right\Vert_{s,\infty}\right) \left\Vert u-v \right\Vert_{s,\infty}
    \end{eqnarray*}
    where $P$ is some quadratic polynomial of two variables with nonnegative coefficients.
    The above results yield the following estimates for  $u,v\in Y(T)$
    \begin{displaymath}
        \Vert Su\Vert_{X(T)}
        \leq \Vert \varphi \Vert_{s,\infty}+T\Vert \psi \Vert_{s,\infty}+C\Vert \alpha \Vert_{1}R^{3}T^{2},
    \end{displaymath}
    and
    \begin{displaymath}
        \Vert Su-Sv\Vert_{X(T)}\leq P\left(R,R\right)\Vert \alpha \Vert_{1} T^{2}\Vert u-v\Vert_{X(T)}
    \end{displaymath}
    concluding the proofs of (\ref{Sestimate}) and (\ref{SLipestimate}).
\end{proof}

\setcounter{equation}{0}
\section{Global Existence and Blow Up in Finite Time}
\label{sec:4}

In this section,we will first show that the maximal time of existence for the solution of the Cauchy problem
(\ref{cau1})-(\ref{cau2}) depends only on the $L^{\infty }$ norm of the initial data. Then we
will prove the existence of a global solution for two classes of  nonlinearities and finally investigate blow-up
for general nonlinearities.

\subsection{Global Existence}
\label{sec:41}

By repeatedly applying local existence theorems (Theorems \ref{theo2.1}-\ref{theo2.4} and \ref{theo3.1})
the solution can be continued to  the maximal time interval $\left[ 0,T_{\max}\right)$ where either
$T_{\max }=\infty $, i.e. we have a global solution, or
\begin{displaymath}
    \lim \sup_{t\rightarrow T_{\max }^{-}} (\Vert u\left( t\right) \Vert_{X}
        +\Vert u_{t}\left( t\right) \Vert_{X})=\infty,
\end{displaymath}
where $\left\Vert {~}\right\Vert _{X}$ denotes either one of the norms in $C_{b}({\Bbb R})$,
    $L^{p}({\Bbb R})\cap L^{\infty }({\Bbb R})$,
     $C_{b}^{1}({\Bbb R})$, $W^{1,p}({\Bbb R})$ or $H^{s}({\Bbb R})\cap L^{\infty }({\Bbb R})$.
\begin{theorem}\label{theo4.1}
    Assume that the conditions in either one of Theorems \ref{theo2.1}-\ref{theo2.4} or \ref{theo3.1} hold.
    Then either there is a global solution or maximal time is  finite, where $T_{\max}$ is characterized by the
    $L^{\infty}$ blow-up condition
    \begin{displaymath}
    \lim \sup_{t\rightarrow T_{\max }^{-}}\Vert u\left( t\right) \Vert_{\infty }=\infty.
    \end{displaymath}
\end{theorem}
\begin{proof}
    Clearly in each case the norm $\left\Vert {~}\right\Vert_{\infty }$ is smaller than
    $\left\Vert {~}\right\Vert _{X}$. Hence it suffices to prove that if
    $\lim \sup_{t\rightarrow T^{-}}\Vert u\left( t\right) \Vert_{\infty}=M <\infty$, then
    $\lim \sup_{t\rightarrow T^{-}}( \Vert u\left( t\right) \Vert_{X}+\Vert u_{t}\left( t\right) \Vert _{X}) <\infty $.
    So assume that the solution exists in some interval $[0,T)$ and satisfies
    $\Vert u\left( t\right) \Vert_{\infty }\leq R$ for all $0\leq t<T$.  The solution satisfies
    \begin{eqnarray*}
        u(x,t) &=&\varphi (x)+t\psi (x)+\int_{0}^{t}(t-\tau)(Ku)(x,\tau )d\tau, \\
        u_{t}(x,t) &=&\psi (x)+\int_{0}^{t}(Ku)(x,\tau )d\tau.
    \end{eqnarray*}
    In all cases the estimate for $Ku$  is of the form
    \begin{displaymath}
        \left\Vert Ku\right\Vert_{X}\leq {\cal M}(\left\Vert u\right\Vert_{\infty})\left\Vert u\right\Vert_{X}
    \end{displaymath}
    with a nondecreasing function $\cal M$ of $\left\Vert u\right\Vert_{\infty}$. Since
    $\Vert u\left( t\right) \Vert_{\infty }\leq R$  for all $t\in [0,T)$,
    \begin{displaymath}
        \left\Vert u\left( t\right) \right\Vert_{X}+\left\Vert u_{t}\left( t\right) \right\Vert_{X}
        \leq \left\Vert \varphi \right\Vert_{X}+\left( 1+T\right) \left\Vert \psi \right\Vert_{X}
            +\left( 1+T\right) {\cal M}(R) \int_{0}^{t}\left\Vert u(\tau) \right\Vert_{X}d\tau,
    \end{displaymath}
    so that Gronwall's Lemma implies
    \begin{displaymath}
        \left\Vert u\left( t\right) \right\Vert_{X}+\left\Vert u_{t}\left( t\right)\right\Vert_{X}
        \leq \left( \left\Vert \varphi \right\Vert_{X}+\left(1+T\right) \left\Vert \psi \right\Vert_{X}\right)
            e^{\left( 1+T\right){\cal M}(R) t}
    \end{displaymath}
    for all $t\in [0,T)$. So
    $\lim \sup_{t\rightarrow T^{-}}( \Vert u\left(t\right) \Vert_{X}+\Vert u_{t}\left( t\right) \Vert_{X})<\infty$.
\end{proof}
\begin{theorem}\label{theo4.2}
    Assume that the conditions in either one of Theorems \ref{theo2.1}-\ref{theo2.4} hold.
    If  the nonlinear term $w$ in (\ref{cau1}) satisfies
    $\left\vert w\left( \eta \right) \right\vert \leq a\left\vert \eta \right\vert +b$ for all $\eta \in {\Bbb R}$, then
    there is a global solution.
\end{theorem}
\begin{proof}
    Assume the solution exists on $[0,T)$. Then
    \begin{eqnarray*}
        \left\vert (Ku)(x,t)\right\vert
        & \leq & \int_{\Bbb R }\left\vert \alpha (y-x)\right\vert
            \left( a\left\vert u(y,\tau)-u(x,\tau )\right\vert +b\right) dy \\
           \hspace*{-60pt}
        & \leq &
           a\left( \left\vert \alpha \right\vert \ast \left\vert u\right\vert \right) (x,t)
            +a\left\Vert \alpha \right\Vert_{1}\left\vert u(x,t )\right\vert
            +b\left\Vert \alpha \right\Vert_{1},
    \end{eqnarray*}
    and by (\ref{inteq})
    \begin{eqnarray*}
        \left\Vert u(t)\right\Vert_{\infty }
        &\leq & \left\Vert \varphi \right\Vert_{\infty}+t\left\Vert \psi \right\Vert_{\infty}
            +\int_{0}^{t}(t-\tau )
            \left( a\left\Vert \left( \left\vert \alpha \right\vert \ast \left\vert u\right\vert \right) (\tau)\right\Vert_{\infty }
            +a\left\Vert \alpha \right\Vert_{1}\left\Vert u(\tau )\right\Vert_{\infty }
            +b\left\Vert \alpha \right\Vert_{1}\right)d\tau \\
        &\leq &\left\Vert \varphi \right\Vert_{\infty}+T\left\Vert \psi \right\Vert_{\infty}
            +bT\left\Vert \alpha \right\Vert_{1}+2aT\left\Vert \alpha \right\Vert_{1}\int_{0}^{t}\left\Vert u(\tau )\right\Vert _{\infty }d\tau ,
    \end{eqnarray*}
    and Gronwall's lemma shows that
    $\lim \sup_{t\rightarrow T^{-}}\Vert u\left(t\right) \Vert_{\infty } <\infty $.
\end{proof}
\begin{lemma}\label{lem4.1}(The Energy Identity)
    Assume that $\alpha \in L^{1}({\Bbb R})$ is  even  and $w\in C^{1}({\Bbb R})$ is odd with $w(0)=0$.
    If $u$ satisfies the Cauchy problem (\ref{cau1})-(\ref{cau2}) on $[0,T)$ with initial data
    $\varphi, \psi \in L^{1} ({\Bbb R}) \cap L^{\infty} ({\Bbb R})$, then the energy
    \begin{displaymath}
        E\left( t\right) =\frac{1}{2}\left\Vert u_{t}\left( t\right) \right\Vert_{2}^{2}
            +\frac{1}{2}\int_{{\Bbb R}^{2}}\alpha (y-x)
            ~W\left(u\left( y,t\right) -u\left( x,t\right) \right) dy dx,
    \end{displaymath}
    is constant for $t\in [0,T)$, where $W(\eta)=\int_{0}^{\eta}w(\rho)d\rho$.
\end{lemma}
\begin{proof}
    By Theorem \ref{theo2.2} with $p=1$ we know $u\in C^{2}([0,T], L^{1}({\Bbb R})\cap L^{\infty}({\Bbb R}))$. Since
    $L^{1}({\Bbb R})\cap L^{\infty}({\Bbb R}) \subset L^{2}({\Bbb R})$, we have $u_{t}(t)\in L^{2}({\Bbb R})$.
    Moreover, an estimate similar to (\ref{Kbound}) where  $w$ is replaced by $W$ shows that the term
    $\alpha(y-x) W(u(y,t)-u(x,t))$ is integrable on ${\Bbb R}^{2}$. Hence $E(t)$ is defined for all $t\in [0,T)$.
    Multiplying (\ref{cau1}) by $u_{t}\left( x,t\right) $ and integrating in $x$ we obtain
    \begin{eqnarray*}
        \int_{\Bbb R}u_{tt}\left( x\right) u_{t}\left( x\right) dx
        & = & \int_{{\Bbb R}^{2}}\alpha \left( y-x\right)
            w\left( u\left( y\right) -u\left( x\right) \right) u_{t}\left( x\right) dy dx \\
        & = &\frac{1}{2}\int_{{\Bbb R }^{2}}\alpha \left(y-x\right)
            w\left( u\left( y\right) -u\left( x\right) \right) u_{t}\left(x\right) dy dx \\
        &&  +\frac{1}{2}\int_{{\Bbb R }^{2}}\alpha \left( y-x\right)
            w\left( u\left( y\right) -u\left( x\right) \right)u_{t}\left( x\right) dy dx,
    \end{eqnarray*}
    where we have again suppressed $t$.
    We now change the order of integration and switch the variables $x,y$ in the last integral to obtain
    \begin{displaymath}
        \frac{1}{2}\int_{{\Bbb R }^{2}}\alpha \left(x-y\right)
            w\left( u\left( x\right) -u\left( y\right) \right) u_{t}\left(y\right) dy dx.
    \end{displaymath}
    Since $\alpha $ is even while $w$ is odd, this gives
    \begin{displaymath}
        -\frac{1}{2}\int_{{\Bbb R }^{2}}\alpha \left(y-x\right)
            w\left( u\left( y\right) -u\left( x\right) \right) u_{t}\left(y\right) dy dx,
    \end{displaymath}
    so that
    \begin{displaymath}
        \int_{\Bbb R }u_{tt}\left( x\right) u_{t}\left( x\right) dx
        =-\frac{1}{2}\int_{{\Bbb R }^{2}}\alpha \left(y-x\right)
        w\left( u\left( y\right) -u\left( x\right) \right) \left(u_{t}\left( y\right) -u_{t}\left( x\right) \right) dy dx.
    \end{displaymath}
    But since $W^{\prime }=w$; we have
    \begin{displaymath}
        \frac{d}{dt}\frac{1}{2}\int_{\Bbb R }\left( u_{t}\left( x\right)\right)^{2}dx
            =-\frac{d}{dt}\frac{1}{2}\int_{{\Bbb R }^{2}}\alpha \left( y-x\right)
            W\left( u\left( y\right)-u\left( x\right) \right) dy dx
    \end{displaymath}
    so that $\frac{dE}{dt}=0$.
\end{proof}
\begin{theorem}\label{theo4.3}
    Assume that $\alpha \in L^{1}({\Bbb R})\cap L^{\infty}({\Bbb R})$ is even with $\alpha \geq 0$  almost everywhere;
    $w\in C^{1}({\Bbb R})$ is odd with $w(0)=0$ and $W \geq 0$.  If there is some $q\geq \frac{4}{3}$ and
    $C>0$ so that
    \begin{equation}
     \left\vert w(\eta) \right\vert^{q}\leq C W(\eta) \label{q}
    \end{equation}
    for all $\eta \in {\Bbb R}$,
    then there is a global solution for initial data  $\varphi, \psi \in L^{1} ({\Bbb R}) \cap L^{\infty} ({\Bbb R})$.
\end{theorem}
\begin{proof}
    Assume that the solution exists in $[0,T)$. By Lemma \ref{lem4.1} the energy is finite and the energy
    identity $E(t)=E(0)$ holds for all $t\in [0,T)$.
    Consider the energy density function
    \begin{displaymath}
        e(x,t) =\frac{1}{2}(u_{t}(x,t))^{2}
            + \int_{{\Bbb R}}\alpha (y-x)~W\left( u\left( y,t\right) -u\left( x,t\right) \right) dy.
    \end{displaymath}
    Differentiating with respect to $t$
    \begin{eqnarray*}
        e_{t}(x,t)
        &=& u_{t}(x,t) u_{tt}(x,t)
            +\int_{\Bbb R}\alpha (y-x)~w( u(y,t)-u(x,t)) (u_{t}(y,t)-u_{t}(x,t)) dy \\
        &=&\int_{\Bbb R}\alpha (y-x)~w(u(y,t)-u(x,t)) u_{t}(y,t) dy.
    \end{eqnarray*}
    Note that by the assumptions of the theorem  $e(x,t)$ and $e_{t}(x,t)$ are in $L^{\infty}({\Bbb R})$
    for each fixed $t$.   Letting $p$ be the dual index to $q$; i.e. $1/p+1/q=1$, we have
    \begin{eqnarray*}
        \left\vert e_{t}\left( x,t\right) \right\vert
        &\leq & \int_{\Bbb R}\alpha (y-x)~\left\vert w(u(y,t)-u(x,t))
            \right\vert \left\vert u_{t}(y,t) \right\vert dy \\
        &\leq & \left\Vert \alpha \right\Vert_{\infty}^{1/p}\left\Vert u_{t}(t)
            \right\Vert_{\infty}^{1-2/p}\int_{\Bbb R}\left\vert u_{t}(y,t)
            \right\vert^{2/p}(\alpha(y-x))^{1/q}~\left\vert w(u(y,t)-u(x,t)) \right\vert dy,
    \end{eqnarray*}
    and by H\"{o}lder's inequality
    \begin{displaymath}
    \left\vert e_{t}(x,t) \right\vert
    \leq \left\Vert \alpha \right\Vert_{\infty}^{1/p}\left\Vert u_{t}(t)
        \right\Vert_{\infty}^{1-2/p}\left(\int_{\Bbb R}\left\vert u_{t}(y,t)\right\vert^{2}dy\right)^{1/p}
        \left( \int_{\Bbb R}\alpha (y-x)~\left\vert w(u(y,t)-u(x,t))\right\vert^{q}dy\right)^{1/q}.
    \end{displaymath}
    Using the condition (\ref{q}) we have
    \begin{displaymath}
        \left\vert e_{t}(x,t) \right\vert
        \leq \left\Vert \alpha \right\Vert_{\infty }^{1/p}\left\Vert u_{t}(t) \right\Vert_{\infty }^{1-2/p}
            \left\Vert u_{t}(t) \right\Vert_{2}^{2/p}\left( C\int_{\Bbb R}\alpha (y-x)~W\left( u\left( y,t\right)
            -u\left( x,t\right) \right)dy\right)^{1/q}.
    \end{displaymath}
    Since $\alpha \geq 0$ and $W\geq 0,$ by the energy identity we have
    $\left\Vert u_{t}\left( t\right) \right\Vert _{2}^{2}\leq 2 E\left( 0\right)$.
    Also, both terms in $e(x,t)$ are nonnegative so that taking essential supremum over $x\in {\Bbb R}$,
    \begin{eqnarray*}
        \left\Vert e_{t}\left( t\right) \right\Vert _{\infty }
        &\leq & \left\Vert \alpha \right\Vert _{\infty }^{1/p}(2E(0))^{1/p}
            \left(2\left\Vert e\left( t\right) \right\Vert_{\infty}\right)^{1/2-1/p}
            (C\left\Vert e \left( t\right) \right\Vert_{\infty })^{1/q} \\
        &\leq &C\left\Vert e \left( t\right) \right\Vert_{\infty }^{r}
    \end{eqnarray*}
    with $r=1/2-1/p+1/q=2/q-1/2$ and some other constant $C$ in the last line. Note that when $q\geq 4/3$,
    $r=\allowbreak 2/q-1/2\leq 1$. Since
    \begin{displaymath}
    e\left( x,t\right) =e\left( x,0\right) +\int_{0}^{t}e_{t}(x,\tau) d\tau
    \end{displaymath}
    we have
    \begin{eqnarray*}
        \left\Vert e\left( t\right) \right\Vert _{\infty }
        &\leq &\left\Vert e(0) \right\Vert _{\infty }+\int_{0}^{t}\left\Vert e_{t}(\tau)
            \right\Vert _{\infty }d\tau  \\
        &\leq &\left\Vert e(0) \right\Vert_{\infty }+C\int_{0}^{t}\left\Vert e(\tau) \right\Vert_{\infty }^{r}d\tau,
    \end{eqnarray*}
    for all  $t\in [ 0,T)$. As $r\leq 1$, we have $\left\Vert e\left( t\right)\right\Vert_\infty^r \leq \left\Vert e\left( t\right)\right\Vert_\infty +1$. By Gronwall's lemma $\left\Vert e\left( t\right) \right\Vert_{\infty }$ and thus $\left\Vert u_{t}\left( t\right) \right\Vert_{\infty}$  stay bounded in $[0,T)$. Integration again gives
    \begin{displaymath}
        \left\Vert u\left( t\right) \right\Vert _{\infty }
        \leq \left\Vert \varphi \right\Vert_{\infty }+\int_{0}^{t}\left\Vert u_{t}(\tau)\right\Vert _{\infty }d\tau
    \end{displaymath}
    so that $\left\Vert u(t) \right\Vert_{\infty }$ does not blow up in finite time.
\end{proof}
\begin{remark}
    Considering the typical nonlinearity $w(\eta)=|\eta|^{\nu -1}\eta$ we have
    $W(\eta)={1\over {\nu +1}}|\eta|^{\nu +1}$. Then the exponent $q$ of Theorem \ref{theo4.3} equals
    $(\nu +1)/\nu$ and $q \ge {4\over 3}$ if and only if $\nu \le 3$. In other words Theorem
    \ref{theo4.3} applies to at most cubic nonlinearities.
\end{remark}

\subsection{Blow-up}
\label{sec:42}
In this section, we will consider the blow-up of the solution for the Cauchy problem
(\ref{cau1})-(\ref{cau2}) by the concavity method. For this purpose, we will use the following
lemma to prove  blow up in finite time.
\begin{lemma}\label{lem4.3} \cite{varga}
    Suppose $H\left( t\right)$, $t\geq 0$ is a positive, twice differentiable function satisfying
    $H^{\prime \prime }(t)H(t)-\left( 1+\nu \right) \left(H^{\prime }(t)\right)^{2}\geq 0$ where $\nu >0$.
    If $H\left( 0\right) >0$ and $H^{\prime }\left( 0\right) >0$, then
    $H\left( t\right) \rightarrow \infty $ as $t\rightarrow t_{1}$ for some
    $t_{1}\leq H\left( 0\right) /\nu H^{\prime }\left( 0\right) $.
\end{lemma}
\begin{theorem}
    Suppose that $\alpha$ is even, $w$ is odd, the conditions of Theorem \ref{theo2.2} hold for $p=1$ and $\alpha \ge 0$  almost everywhere.
    If there is some   $\nu >0$ such that
    \begin{displaymath}
        \eta w\left( \eta\right) \leq 2\left( 1+2\nu \right) W\left( \eta\right) ~\mbox{ for all }~\eta\in {\Bbb R},
    \end{displaymath}
    and
    \begin{displaymath}
        E\left( 0\right) =\frac{1}{2}\left\Vert \psi \right\Vert_{2}^{2}
            +\frac{1}{2}\int_{{\Bbb R }^{2}} \alpha(y-x)~W\left(\varphi(y) -\varphi(x)\right) dy dx<0,
    \end{displaymath}
    then the solution $u$ of the Cauchy problem (\ref{cau1})-(\ref{cau2}) blows up in finite time.
\end{theorem}
\begin{proof}
    Assume that there is a global solution. Then
    $u(t), u_{t}(t)\in L^{1}({\Bbb R})\cap L^{\infty}({\Bbb R})\subset L^{2}({\Bbb R})$ for all $t>0$. Let
    $H\left( t\right) =\left\Vert u(t)\right\Vert_{2}^{2}+b\left( t+t_{0}\right) ^{2}$ for some positive constants $b$
    and  $t_{0}$ to be determined later. Suppressing the $t$ variable throughout the computations
    \begin{eqnarray*}
        H^{\prime }(t)  &=&2\left\langle u,u_{t}\right\rangle+2b\left( t+t_{0}\right)  \\
        H^{\prime \prime }(t) &=&2\left\Vert u_{t}\right\Vert_{2}^{2}
            +2\left\langle u,u_{tt}\right\rangle +2b.
    \end{eqnarray*}
    Using (\ref{cau1})
    \begin{eqnarray*}
        2\left\langle u,u_{tt}\right\rangle &=& 2
        \int_{{\Bbb R }^{2}} \alpha(y-x)~w\left( u(y)-u(x)\right)u(x) dy dx
        \nonumber \\
        &=& \int_{{\Bbb R }^{2}}\alpha \left(y-x\right)
            w\left( u\left( y\right) -u\left( x\right) \right) u\left(x\right) dy dx \\
        &&  +\int_{{\Bbb R }^{2}}\alpha \left( y-x\right)
            w\left( u\left( y\right) -u\left( x\right) \right)u\left( x\right) dy dx.
        \end{eqnarray*}
    Interchanging the variables $x$ and $y$ in the second  integral and noting that $\alpha$ is even and $w$ is odd
    we get
    \begin{eqnarray*}
        2\left\langle u,u_{tt}\right\rangle
        &=& \int_{{\Bbb R }^{2}}  \alpha(y-x)~w\left( u(y)-u(x)\right) u(x) dy dx \\
        && -\int_{{\Bbb R }^{2}}  \alpha(y-x)~w\left( u(y)-u(x)\right) u(y) dy dx  \\
        &=& -\int_{{\Bbb R }^{2}}  \alpha(y-x)~w\left( u(y)-u(x)\right) (u(y)-u(x)) dy dx.
    \end{eqnarray*}
    So that
    \begin{eqnarray*}
     2\left\langle u,u_{tt}\right\rangle
        &\geq & -2\left( 1+2\nu \right) \int_{{\Bbb R }^{2}} \alpha(y-x)~W\left( u(y)-u(x)\right) dy dx \\
        &=&4\left( 1+2\nu \right) ( \frac{1}{2}\left\Vert u_{t}\right\Vert_{2}^{2}-E( 0) ).
    \end{eqnarray*}
    Hence we get
    \begin{displaymath}
        H^{\prime \prime }\left( t\right)
            \geq 4\left( 1+\nu \right)  \left\Vert u_{t}\right\Vert_{2}^{2}-4\left( 1+2\nu \right)E\left( 0\right)+2b.
    \end{displaymath}%
    On the other hand, we have
    \begin{eqnarray*}
        \left( H^{\prime }\left( t\right)\right)^{2}
        &=&4 \left[\left\langle u,u_{t}\right\rangle+b\left( t+t_{0}\right) \right]^{2}  \\
        &\leq &4\left[ \left\Vert u\right\Vert_{2}\left\Vert u_{t}\right\Vert_{2}
                +b\left( t+t_{0}\right)\right]^{2} \\
        &= &4\left[ \left\Vert u\right\Vert_{2}^{2}\left\Vert u_{t}\right\Vert_{2}^{2}
                +2\left\Vert u\right\Vert_{2} \left\Vert u_{t}\right\Vert_{2} b\left( t+t_{0}\right)
                +b^{2}\left( t+t_{0}\right)^{2}\right)] \\
        &\leq &4\left[ \left\Vert u\right\Vert_{2}^{2}\left\Vert u_{t}\right\Vert_{2}^{2}
                +b\left\Vert u\right\Vert_{2}^{2}
                +b\left\Vert u_{t}\right\Vert_{2}^{2}\left( t+t_{0}\right)^{2}
                +b^{2}\left( t+t_{0}\right)^{2}\right].
    \end{eqnarray*}
    Thus
    \begin{eqnarray*}
        H^{\prime \prime }\left( t\right) H\left( t\right) &-&\left( 1+\nu \right)
            \left( H^{\prime }\left( t\right) \right)^{2} \\
        &\geq & \left[ 4\left( 1+\nu \right)  \left\Vert u_{t}\right\Vert_{2}^{2}-4\left( 1+2\nu \right)E(0)+2b\right]
             \left[ \left\Vert u\right\Vert_{2}^{2}+b\left( t+t_{0}\right)^{2}\right] \\
        &&-4\left( 1+\nu \right) \left[ \left\Vert u\right\Vert_{2}^{2} \left\Vert u_{t}\right\Vert_{2}^{2}
            +b \left\Vert u\right\Vert_{2}^{2}+b\left\Vert u_{t}\right\Vert_{2}^{2} \left( t+t_{0}\right)^{2}
            +b^{2}\left( t+t_{0}\right)^{2}\right]\\
        &=&\left[-4\left( 1+2\nu \right) E(0)+2b-4b\left( 1+\nu \right)\right]
           \left[ \left\Vert u\right\Vert_{2}^{2}+b\left( t+t_{0}\right)^{2}\right]  \\
        &=&-2\left( 1+2\nu \right) \left(b+2E(0)\right) H(t).
    \end{eqnarray*}
    Now if we choose $b\leq -2E(0)$, this gives
    \begin{displaymath}
        H^{\prime \prime }\left( t\right) H\left( t\right)
            -\left( 1+\nu \right)\left( H^{\prime }\left( t\right) \right) ^{2}\ge 0.
    \end{displaymath}
    Moreover
    \begin{displaymath}
        H^{\prime }\left( 0\right) =2\left\langle \varphi ,\psi \right\rangle+2bt_{0} > 0
    \end{displaymath}
    for sufficiently large $t_{0}$. According to  Lemma \ref{lem4.3}, this implies
    that $H\left( t\right)$, and thus $\left\Vert u(t)\right\Vert_{2}^{2}$ blows
    up in finite time contradicting the assumption that the global solution exists.
\end{proof}

\noindent
{\bf Acknowledgement}: This work has been supported by the Scientific and Technological Research Council of Turkey
(TUBITAK) under the project TBAG-110R002.

\end{document}